\documentclass{amsart}
\usepackage[all]{xy}
\usepackage{palatino, 
            mathpazo, 
            amssymb, amsthm, epsfig, url}
\begin{document}

\newtheorem{thm}{Theorem}[section]
\newtheorem{prop}[thm]{Proposition}
\newtheorem{cor}[thm]{Corollary}
\newtheorem{lemma}[thm]{Lemma}
\newtheorem{rmk}[thm]{Remark}
\newtheorem{const}[thm]{Construction}
\newtheorem{defn}[thm]{Definition}
\newtheorem{exa}[thm]{Example}
\newtheorem*{cmptness_thm}{Compactness Theorem for Finitely Presented Groups}
\newtheorem*{ectcg}{Effective Compactness Theorem for Coxeter groups}

\newcommand\gp[2]{\langle\,{#1}\mid{#2 }\,\rangle}
\newcommand\genby[1]{\langle {#1} \rangle}

\renewcommand\H{\mathbb{H}}
\newcommand{\R} { {\mathbb R} }
\newcommand{\D}{\mathcal D}
\newcommand\Hom{\textnormal{Hom}}
\newcommand\Isom{\textnormal{Isom}}
\newcommand\stab{\textnormal{stab}}
\newcommand\vx{\textnormal{vert}}
\newcommand\edge{\textnormal{edge}}
\newcommand\df{\textnormal{df}}
\newcommand\Thyp{T_{\mathbb{H}}}
\newcommand\Lab{\textnormal{Lab}}
\newcommand\newLab{\overline{\Lab}}
\newcommand\sh{\textnormal{sh}}
\newcommand\Tsh{T_\sh}
\newcommand\spl{\textnormal{spl}}
\newcommand\Tspl{T_\spl}
\newcommand\Labspl{\Lab_\spl}

%

\title{An Effective Compactness Theorem for Coxeter Groups}
\author{Yvonne Lai}
\email{yxl@umich.edu}
\address{Department of Mathematics \\ University of Michigan \\ Ann Arbor MI, 48104}
\thanks{This material is based upon work supported 
by the National Science Foundation Grants DMS-0135345, 
DMS-05-54349, DMS-04-05180, and DMS-0602191.}


\begin{abstract}
Through highly non-constructive methods, works by Bestvina, Culler, Feighn, Morgan, Rips, Shalen, and Thurston show that if a finitely presented group does not split over a virtually solvable subgroup, then the space of its discrete and faithful actions on $\H^n$, modulo conjugation, is compact for all dimensions. Although this implies that the space of hyperbolic structures of such groups has finite diameter, the known methods do not give an explicit bound. We establish such a bound for Coxeter groups. We find that either the group splits over a virtually solvable subgroup or there is a constant $C$ and a point in $\H^n$ that is moved no more than $C$ by any generator. The constant $C$ depends only on the number of generators of the group, and is independent of the relators.
\end{abstract}

\date{}

\maketitle

\section*{Introduction}

The space of discrete and faithful actions of a given group $G$ on $\H^n$, up to
conjugation, is a {\it deformation space} of the group.  It is denoted $\D(G,n)$.
In the 1980's, Thurston proved that when a group $G$ is the fundamental group
of an orientable, compact, irreducible, acylindrical 3-manifold with boundary, 
the deformation space $\D(G,3)$ is compact \cite{Thu86}.  To prove this
result, Thurston analysed sequences of ideal triangulations.  

Inspired by Thurston's work and Culler-Shalen's work on varieties
of three-manifold groups \cite{CS83}, Morgan-Shalen reproved Thurston's compactness
theorem using methods from algebraic geometry and geometric topology \cite{MS84} \cite{MS88a} \cite{MS88b}.  
Morgan then showed that when $G$ 
is the fundamental group for a compact, orientable, and irreducible 3-manifold, the space $\D(G,n)$
is compact if and only if the group $G$ does not admit a virtually 
abelian splitting \cite{Mo86}.  This result was pushed to include
all finitely-presented groups using the Rips Machine 
by Bestvina-Feighn,
who state the following Compactness Theorem as a consequence of the
main result of \cite{BF95} concerning actions of trees:
\begin{cmptness_thm}[Thurston, Morgan-Shalen, Morgan, Rips, Bestvina-Feighn]
If $G$ be a finitely-presented group that is not virtually abelian
and does not split over a virtually solvable subgroup, then $\D(G,n)$ 
is compact.  
\end{cmptness_thm}

If a finitely-presented group does not split over a virtually solvable
subgroup, then the Compactness Theorem implies that there is a point
in $\H^n$ that is not moved too far by any generator, for any 
action by the group. 
However, the methods in \cite{BF95} and \cite{Mo86} do
not give an explicit bound.
The technical 
adjective \textit{ineffective} describes such non-constructive results.
In contrast, if a proof is constructive or yields explicit quantities,
then it is termed \textit{effective}.
The main result of this paper gives an estimate for this uniform bound 
in the case of Coxeter groups, in terms of the displacement function.

Given a finite presentation
of a group $G$ with generators $\{g_i\}$, and a representation
$\rho: G\to \Isom(\H^n)$,
we define the displacement function of the
action corresponding to $\rho$
as the ``mini-max'' function
$\inf_{x\in\H^n}\{\max_i d(g_ix, x)\}$.

\begin{ectcg}
Let $G$ be a Coxeter group given by a standard presentation
with $k$ generators, and suppose that $G$ admits an isometric discrete and
faithful action
on $\H^n$.  There exists a function $C_n(k)\in O(k^4)$
so that either $G$ has a virtually solvable special nontrivial splitting or the 
displacement function is bounded above by $C_n(k)$
for every discrete and faithful action of $G$ on $\H^n$.
(We state the function explicitly in
Section \ref{s: small decompositions}.)
\end{ectcg}

This result is related to work by Delzant \cite{Del95} and Barnard \cite{Bar07}.
Delzant \cite{Del95} proved an effective compactness theorem 
for faithful representations of groups to Gromov-hyperbolic groups.  Barnard
\cite{Bar07} proved an effective compactness theorem for surface groups
acting on an arbitrary complete geodesic $\delta$-hyperbolic space, which
generalizes the Mumford Compactness Theorem to $\delta$-hyperbolic spaces.
Both Delzant and Barnard's results rely on the assumption that the injectivity
radius of the group action is bounded from below.  (The result in this
paper does not use such an assumption.)

\subsection*{Summary}
We begin by recalling the definitions and properties related to
Coxeter groups that we use. Section \ref{s: coxeter} reviews 
special subgroups of Coxeter groups
and defines special splittings following Mihalik and Tschantz's visual
decompositions \cite{MT}.

We then discuss the hyperbolic geometry lemmas needed for the result.
In Section \ref{s: bounding midpoints}, we calculate an estimate $\Lambda(\epsilon, R)$
for the length of a geodesic segment in $\H^n$ that guarantees
that the midpoint of the segment is moved at most $\epsilon$ by an involution, if 
the translation distance for the endpoints is bounded above by a constant
$R$. 
In Section \ref{s: qc hull}, we show that the quasi-convex hull of
a finite set $X$ in $\H^n$ is quasi-isometric to the Gromov approximating
tree for $X$, which is an abstract tree. 
In Section \ref{s: shadow}, we construct a projection of
the tree to a collection of geodesic segments in $\H^n$ 
spanning the set $X$, called the ``shadow'' of the
Gromov approximating tree.  To show that the shadow is 
quasi-isometric to a Gromov approximating 
tree of $X$, we use the quasi-isometry from Section \ref{s: qc hull}.

To relate special splittings to the geometry of the action, 
we describe a
combinatorial framework for assigning
labels to the vertices of a tree.
Section \ref{s: labelling systems information} introduces the system by which
we label vertices.
In Section \ref{s: bounds on B}, we use the fixed points of a Coxeter
group action to generate a Gromov approximating tree. 
We apply the labelling system to
the Gromov approximating 
tree to produce splittings of the Coxeter group.  Each edge of
the tree yields a special splitting.  

In Section \ref{s: small decompositions}, 
we combine the above to prove the main result.  
We show that when an edge of the Gromov approximating tree is
sufficiently long, then the splitting produced by an edge
is nontrivial and small.
Given an action $\rho$, we find a lower
bound on the displacement function of $\rho$ that ensures
that the associated Gromov tree contains such a sufficiently long edge.
To do so, we apply the estimate $\Lambda(\epsilon, R)$ obtained in Section \ref{s: bounding midpoints}
to geodesics contained in the shadow of the approximating tree,
setting $\epsilon$ to the Margulis constant for $\H^n$.

\subsection*{Acknowledgments}
I am grateful for the insight and mentorship of my advisor
Misha Kapovich.  I am also thankful for helpful conversations and
encouragement from Angela Barnhill, Dick Canary, Indira Chatterji, 
Thomas Delzant, and Jean Lafont.

\section{Coxeter groups}
\label{s: coxeter}

\subsection{Notation}
We begin by laying out conventions that are used in this paper:
\vskip16pt

\begin{center}
  \begin{tabular}{ll}
   $W$ & Coxeter group \\
   $S$ & generating set for $W$ in a standard presentation \\
   $(W,S)$ & Coxeter system \\ 
   $\Gamma(W,S)$ & Coxeter diagram for system $(W,S)$    
  \end{tabular}
\end{center}

\vskip16pt  
\noindent The terms above are defined in 
Section \ref{s: overview Cox}.

\subsection{Overview of Coxeter groups}
\label{s: overview Cox}

We briefly recall definitions and and properties
of Coxeter groups that are used in the remainder of the paper.
The most relevant Coxeter group properties for this paper are: (1) the generators
in the standard presentations have finite order (see Definition \ref{d: Cox gp}), and (2) each relator
corresponds to a finite subgroup of $W$ (see Remark \ref{r: Cox gp}).
The length of a relator is inconsequential.

\begin{defn}
\label{d: Cox gp}
A group $W$ is a {\em Coxeter group} if
it admits a presentation of the form
  $$\gp{s_1,\ldots,s_g}{(s_is_j)^{m_{ij}}}.$$
where $m_{ij} \leq \infty$ and 
  \begin{itemize}
  \item $m_{ij}=1$ if and only if $i=j$
  \item $m_{ij}\geq 2$ when $i>j$.
  \item $m_{ij}=\infty$ if and only if the element $s_is_j$ has infinite order.
  \end{itemize}
We call this presentation a {\em standard presentation}.

Denote the set of generators $S=\{s_1,\ldots,s_g \}$.  
We call the pair $(W,S)$ a 
{\em Coxeter system} and $S$ a {\em fundamental set of generators}.
We say that the {\em rank} of a Coxeter system is the cardinality of $S$.
\end{defn}  
\begin{rmk}
\label{r: Cox gp}
If $(s_is_j)^{m_{ij}}$ is a relator in a presentation of $W$, then $\genby{s_i,s_j}$
is a finite dihedral group.
\end{rmk}

Later in the paper, we may abbreviate {\em Coxeter system} as {\em system}.

A {\em Coxeter diagram} is a graph that encodes the information
given by a standard presentation of a Coxeter group.
We denote the graph as $\Gamma(W,S)$.  
Its vertices correspond to
the generators bijectively; each vertex is labelled with its corresponding
generator.  An edge exists between two vertices $s_i$ and $s_j$
if and only if $m_{ij}$ is finite and $i\neq j$. This edge is labelled
with the number $m_{ij}$.   Every Coxeter group determines a Coxeter diagram,
a graph whose edges are labelled by positive integers, and any such
graph determines a Coxeter system.

Figure \ref{f: sample Coxeter diagram} shows 
the diagram corresponding
to the reflection group about a hyperbolic quadrilateral
with angles $\pi/2$, $\pi/3$, $\pi/4$, $\pi/4$.  A standard
Coxeter presentation for this group is 
$$\gp{s_1,s_2,s_3,s_4}{s_i^2,(s_1s_2)^2,
                      (s_2s_3)^4,(s_3s_4)^3,(s_4s_1)^4}.$$

\begin{figure}
\includegraphics{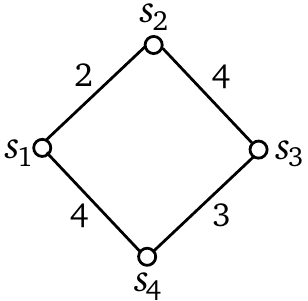}
\caption{\label{f: sample Coxeter diagram} A sample Coxeter diagram.}                      
\end{figure}

Note that if a Coxeter diagram $\Gamma(W,S)$ is disconnected,
then the Coxeter group can be expressed as a free product of
the groups given by the components.  For the remainder
of the paper, we work only with Coxeter groups whose diagrams
are connected. This assumption will be used 
in Section \ref{s: application}.  

By abuse of notation, we may use $s_i$ to denote the vertex of
the Coxeter diagram that is labelled by the generator $s_i$
as well as to denote the generator.

Note that the Coxeter diagram convention differs from the {\em Coxeter graph}, 
where edges are drawn if and only if $2<m_{ij}\leq \infty.$  The Coxeter
group associated to a 
disconnected Coxeter graph 
can be decomposed into a direct product.

Coxeter diagrams are more common in geometric group theory,
while Coxeter graphs are more common in Lie theory and combinatorics.

\subsection{Special splittings}
\label{s: special}

\begin{defn}
\label{d: special}
Given a Coxeter system $(W,S)$, a subgroup $W'$ of $W$ is {\em special}
if it is generated by a subset of $S$.   The notation
$(W',S')\subset (W,S)$ indicates that the subgroup $W'$ is generated by 
$S'\subset S$.  

As a set of vertices, a subset $S'$ spans a unique maximal subdiagram 
$\Gamma'=\Gamma(W')$ 
of the Coxeter
diagram $\Gamma(W)$ of $(W,S)$.  
We say that the subdiagram $\Gamma'$ is {\em special}
and call the associated subgroup $W(S')=\genby{S'}$.   

A Coxeter system $(W'',S'')$ is isomorphic to a special subgroup of $(W,S)$
if there is an injection $j:W''\hookrightarrow W$ carrying $S''$ to a subset of $S$.
We call $j$ a {\em special} injection.
\end{defn}

Recall that when a group can be expressed as the amalgamated product $A*_C B$
or HNN-extension $A*_C$,
we say that the group {\em splits} over $C$, and we refer to the amalgamated product or HNN-extension as a {\em splitting}.  In this paper, we will not need
to consider HNN-extensions.  We let the injections defining an amalgamated product be denoted $i_A:C\hookrightarrow A$ and $i_B: C\hookrightarrow B$.
A presentation of an amalgamated product is given by
   $$A*_C B=\left\langle S(A)\cup S(B) \;\; \bigg{\vert} \;\;
        R(A) \cup R(B) \cup \bigcup_{s\in S(C)} \{ i_A(s)=i_B(s) \}  
        \right\rangle,$$
where the group $A$ is given by the presentation $\gp{S(A)}{R(A)}$, the group $B$ is
given by the presentation $\gp{S(B)}{R(B)}$, and amalgamation subgroup 
$C$ is generated by $S(C)$.
 
\begin{defn}
\label{d: special splitting}
A splitting $W=A*_C B$ is a {\em special splitting} of a Coxeter system $(W,S)$
if the following conditions hold:
\begin{itemize}
\item $A$, $B$, and $C$ are special subgroups of $W$.
\item The Coxeter diagram $\Gamma(C)$ is a subdiagram of 
$\Gamma(A)$ and $\Gamma(B)$.
\item 
Let  $j_A: A\hookrightarrow W$,
$j_B: B\hookrightarrow W$, $j_C: C\hookrightarrow W$
be special injections for $A$, $B$, $C$ into the Coxeter system $(W,S)$. 
Let $\Gamma_A$, $\Gamma_B$, and $\Gamma_C$ 
be the subgraphs of $\Gamma(W)$
induced by the images of the special injection maps.  
Then the amalgamation maps $i_A: C\hookrightarrow A$ and $i_B:C\hookrightarrow B$
are induced by the inclusions of the subgraph $\Gamma_C$ into $\Gamma_A$
and $\Gamma_B$.
\end{itemize}
\end{defn}

The conditions in Definition \ref{d: special splitting} 
are Mihalik and Tschantz's {\em visual axioms} for the case
of splittings, or graphs of groups with one edge. 
The visual axioms for graphs of groups decompositions of Coxeter groups
were introduced by Mihalik and Tschantz in \cite{MT}, where the authors
used special decompositions to show accessibility with respect to 
2-ended splittings and to classify maximal FA-subgroups of finitely
generated Coxeter groups.
        
\begin{defn}
\label{d: trivial splitting}
A splitting is {\em trivial} if one of the amalgamation maps $i_A$ or $i_B$
is an isomorphism.
\end{defn}

When the amalgamation groups $A$ and $B$ are Coxeter groups,
and $C$ is a special subgroup of $A$ and $B$, then the group $A *_C B$ is a Coxeter
group as well.  Its diagram can be obtained
by ``visually amalgamating'' the Coxeter diagrams for $A$ and $B$ (as graphs), in the following manner:
\begin{defn}
\label{d: visual amalgamation}
 Suppose that 
$\alpha_A:\Gamma_C\hookrightarrow \Gamma_A, \; \alpha_B:\Gamma_C\hookrightarrow \Gamma_B$
are injections of the labelled simplicial graph $\Gamma_C$ 
carrying edges to edges, vertices to vertices, forgetting vertex labels, 
and such that labels on edges are preserved.
Then the labelled graph
    $$\Gamma=\Gamma_A \cup \Gamma_B / \sim,$$
where $x\sim y$ when $x=\alpha_A\circ \alpha_B^{-1}(y)$
is called the {\em visual amalgamation}
of the diagrams $\Gamma_A$ and $\Gamma_B$ over $\Gamma_C$.  
The edges of $\Gamma$ inherit labels from $\Gamma_A$ and $\Gamma_B$,
and the vertices of $\Gamma$ are unlabelled. We write $\Gamma=\Gamma_A\cup_{\Gamma_C} \Gamma_B.$ 
\end{defn}

Given Coxeter systems $(W_A,S_A)$, $(W_B,S_B)$, $(W_C, S_C)$, 
let $\Gamma_A$, $\Gamma_B$, $\Gamma_C$ be their associated Coxeter diagrams. 
Suppose that there are special injections from $W_C$ to $W_A$ and $W_B$ given by $j_A:(W_C,S_C)\hookrightarrow (W_A,S_A)$ and
$j_B:(W_C,S_C)\hookrightarrow (W_B,S_B)$.
Let $L_A:S_A\to \vx(\Gamma_A)$ be the bijection sending $s\in S_A$ to the vertex in $\Gamma_A$ labelled $s$, 
and similarly for $B$.  Let $L_C$ be the bijection between $S_C$
and vertices of $\Gamma_C$. Let $\alpha_A$ and $\alpha_B$ be defined
as in Definition \ref{d: visual amalgamation}.
Then the following diagram commutes:

\begin{displaymath}
\xymatrix{
\vx(\Gamma_A)\ar@{{<}->}[ddd]_{L_A}& \amalg & 
            \vx(\Gamma_B)\ar@{{<}->}[ddd]^{L_B} \\
            & \vx(\Gamma_C)\ar@{{<}->}[d]^{L_C}\ar@{^{(}->}[lu]^{\alpha_A}\ar@{_{(}->}[ru]_{\alpha_B}&  \\
           &    S_C\ar@{{<}->}[u]\ar@{_{(}->}[ld]^{j_A}\ar@{^{(}->}[rd]_{j_B}   &     \\
         S_A            &  \cup              &  S_B.       \\                                    
}
\end{displaymath}

\begin{prop}
\label{p: visual amalgamation of two groups}
Let $W=W_A*_{W_C} W_B$, where the amalgamation is given by the maps $j_A$ and $j_B$.   Let $(W',S')$ be the Coxeter system defined by the diagram $\Gamma$.
Then $W\cong W'$.
\end{prop}

\begin{proof}  
By inspection of the commutative diagram following Definition \ref{d: visual amalgamation}.
\end{proof}

\begin{defn}
\label{d: determines a special splitting}
Let $S'$ be a subset of generators for the system $(W,S)$.
Suppose that there are subgroups $W_A\subset W$ and $W_B\subset W$,
and maps  
$i_A: \genby{S'}\hookrightarrow W_A$ and $i_B: \genby{S'}\hookrightarrow W_B$
so that the amalgamated product $W_A*_{\genby{S'}} W_B\cong (W,S)$ 
determined by $i_A$ and $i_B$ is a special splitting. Then 
we say that the subset $S'$ {\em determines
a special splitting} of $W$.
\end{defn}
\begin{exa}
Suppose $(W,S)=\gp{s_1,s_2, s_3, s_4}{\{s_i^2\}, \{(s_is_{i+1})^{m_{i,i+1}}\}}$
as in Figure \ref{fig: sample splitting}.  
The subset $S'=\{s_1,s_3\}$ determines
the splitting 
$$\genby{s_1,s_2,s_3}*_{\genby{s_1,s_3}} \genby{s_1,s_3,s_4}.$$
\end{exa}

\begin{figure}
\includegraphics{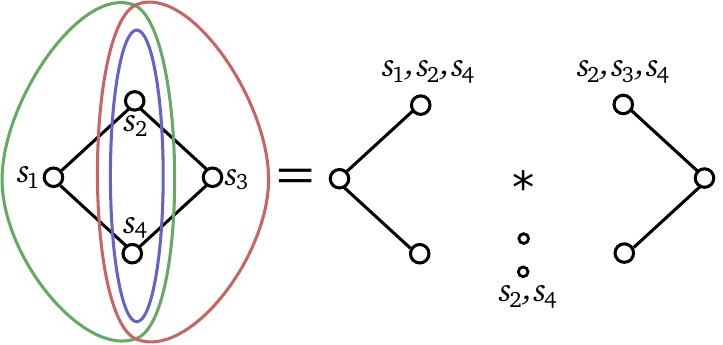}
\caption{\label{fig: sample splitting} A sample splitting.}
\end{figure}

\begin{rmk}
\label{r: splitting not unique}
A subset does not always determine a unique special splitting, even when $\Gamma$ is connected.
\end{rmk}

\begin{prop}
\label{prop: special splittings separate Gamma}
Suppose that a Coxeter system $(W,S)$ contains a special
subgroup $W'$ with diagram $\Gamma'\subset \Gamma$.  Then the
subgroup $W'$ determines a special splitting of the system $(W,S)$
if and only if the subdiagram $\Gamma'$ separates the diagram $\Gamma$. 
\end{prop}

\begin{proof}
Suppose that $\Gamma'$ separates $\Gamma$.  Then there exist
open nonempty disjoint subsets $\Gamma_A'$ and $\Gamma_B'$ of 
$\Gamma\setminus \Gamma'$ that cover $\Gamma\setminus \Gamma'$.
Let $\Gamma_A$ be the maximal subgraph of $\Gamma$ spanned by 
$\vx(\Gamma'\cup\Gamma_A')$, and set $W_A=\genby{\vx(\Gamma_A)}$, and
similarly for $B$.
By Proposition \ref{p: visual amalgamation of two groups}, the Coxeter system
corresponding to the diagram
$\Gamma=\Gamma_A\cup_{\Gamma'}\Gamma_B$ is $W_A*_{W'} W_B\cong W$.

Conversely, suppose that $W'$ determines a special splitting. 
Let $W_A *_{W'} W_B$ be a special splitting of $W$.
By Proposition \ref{p: visual amalgamation of two groups},
the Coxeter diagram for $W$ is given by $\Gamma_A\cup_{\Gamma'} \Gamma_B$
with the amalgamation maps induced by the identity inclusion.  Hence
$\Gamma_A\cap \Gamma_B=\Gamma'$, so $\Gamma'$ separates $\Gamma$.
\end{proof}

We are ultimately
interested in nontrivial special splittings.  Recall from Definition \ref{d: trivial splitting} that 
 trivial splitting occurs when at least one of the groups $W_A$ or $W_B$
equals $W$.  This is the case if and only if $i_A:\genby{S_C}\hookrightarrow W_A$ or $i_B:\genby{S_C}\hookrightarrow W_B$
is an isomorphism, so one of $S_A$ or $S_B$ is the entire
set $S$.
Thus a trivial splitting occurs when one
of the subdiagrams $\Gamma_A$ or $\Gamma_B$ is the entire diagram $\Gamma$.

\section{Bounding the movement of midpoints}
\label{s: bounding midpoints}

The main result of this section is Proposition \ref{p: epsilon-R estimate},
which finds an estimate $\Lambda(\epsilon, R)$ for the length of
a geodesic segment $e$ in $\H^n$ that guarantees the following:
if an isometric involution of $\H^n$ moves the endpoints of $e$ at most
distance $R$,
then the midpoint of $e$ is moved at most $\epsilon$.

Recall that a group is {\it virtually} $P$ if it contains 
a finite index subgroup with property $P$.

\begin{defn}
\label{d: small group}
A group is {\em small} if it is virtually solvable.
\end{defn}

\begin{thm}[Kazhdan-Margulis Theorem, \cite{KM68}]
\label{t: K-M}
There exists a constant $\mu_n>0$ (called the {\em Margulis constant}
for $\H^n$) 
with the following property. 
Let $x\in \H^n$ and $G$ be a discrete subgroup of $\Isom(\H^n)$
generated by $\{g_j\}$ such that
$$d(x,g_j(x))\leq \mu_n$$ for all $j$.  Then the group $G$
is small.
\end{thm}  

\begin{prop}
\label{p: epsilon-R estimate}
(See Figure \ref{f: kap-quad-lemma}.)
Set $R\geq0$.  Let $e=[v_1v_2]$ be a geodesic segment
in $\H^n$ and let $s$ be an isometric involution of $\H^n$.
Suppose that they satisfy
$$d(v_1,s(v_1)),d(v_2,s(v_2))\leq R.$$

Let $m$ denote the midpoint of $e$.  
Define $\Lambda(\epsilon,R)$ as
    \begin{equation}
    \label{e: Lambda}
    \Lambda(\epsilon, R)
      =\frac{4}{\epsilon}
       +2R.
    \end{equation}
    Then for every $\epsilon>0$,
if $d(v_1,v_2)\geq \Lambda(\epsilon,R)$, then $d(m,s(m))\leq \epsilon$.

\end{prop}

\begin{figure}
\includegraphics{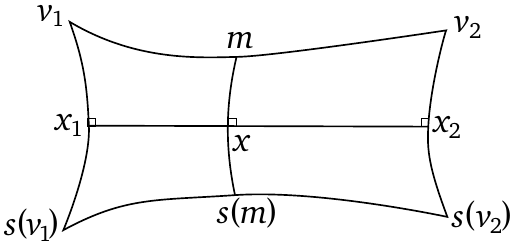}
\caption{\label{f: epsilon-R estimate} Diagram for Proposition \ref{p: epsilon-R estimate}.}
\end{figure}
\begin{figure}
\includegraphics{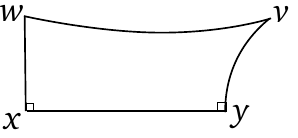}
\caption{\label{f: kap-quad-lemma} Diagram for Lemma \ref{l: kap quad lemma}.}
\end{figure}

The proof of Proposition \ref{p: epsilon-R estimate} relies on the 
convexity of the distance function in hyperbolic space
via the propositions that follow.

Here and in future sections, we use $[xy]_\H\subset \H^n$ 
to denote the geodesic path between two points $x$ and $y$ in
$\H^n$.   We denote the triangle $[xy]_\H\cup[yz]_\H\cup[zx]_\H$ by 
$[xyz]_\H$ and define a quadrilateral $[xyzw]_\H$ similarly.  A quadrilateral
$[xyzw]_\H$ may not necessarily be planar.  Where there is no ambiguity,
we may suppress subscripts.

\begin{lemma}[{Convexity of the hyperbolic distance function
\cite[II.2, Proposition 2.2]{BH}}].
\label{c: convexity of distance}
Let $X$ be a geodesic metric space, and $c:[0,1]\to X$ and $c':[0,1]\to X$ be
two geodesics parametrized proportionally to arc length.    Then for any $t,t'\in [0,1]$,
the maps $c$ and $c'$ satisfy the inequality
    $$d(c(t),c'(t))\leq (1-t)d(c(0),c'(0))+td(c(1),c'(1)).$$
\end{lemma}

The following is an immediate consequence of Lemma \ref{c: convexity of distance}.

\begin{cor}
\label{c: convexity corollary}
Define $c$ and $c'$ as in Lemma \ref{c: convexity of distance}, let $I$ denote the interval $[0,1]$, 
and let $C'=c'(I)$.
Suppose $d(c(0),c'(0)),d(c(1),c'(1))\leq r$.  Then $d(c(t),C')\leq r$ for all $t\in I$.
\end{cor}

To obtain the estimate in Proposition \ref{p: epsilon-R estimate}, we use the function
$$h(x)=\sinh^{-1}\left( \frac{1}{\sinh(x)}\right),$$
where $x\in\R$ is strictly positive. 

\begin{lemma}[{\cite[Lemma 3.5, pp. 34-35]{Kbook}}]
\label{l: kap quad lemma}
(See Figure \ref{f: kap-quad-lemma}.) 
Let $[xyvw]$ be a quadrilateral in $\H^n$ with angles
$\angle wxy =\frac{\pi}{2}$, $\angle xyv=\frac{\pi}{2}$,
$\angle xwv \geq \frac{\pi}{2}$. Then 
 $$d(x,w)\leq h(d(x,y)).$$
\end{lemma}

\begin{cor}
\label{l: epsilon quad corollary}
Fix $\epsilon>0, R\geq0$.
Let $[xyvw]$ be defined as in Lemma \ref{l: kap quad lemma}.
Suppose that $$d(x,w),d(y,v)\leq \frac{R}{2} \textrm{ and } 
d(v,w)\geq h^{-1}\left(\frac{\epsilon}{2}\right)+R.
$$
Then 
$d(x,w)\leq \frac{\epsilon}{2}$.
\end{cor}

\begin{proof}[Proof of Lemma \ref{l: epsilon quad corollary}]
We first note that $h$ is decreasing.  Let $g=h^{-1}(\frac{\epsilon}{2})+R$.  Since
$d(v,w)\geq g$, we have $$d(x,y)\geq g-(d(x,w)+d(y,v))\geq g-R$$
and $h(d(x,y))\leq h(g-R)$.  We conclude that 
$$d(x,w)\leq h(d(x,y))\leq h(g-R)
   =h\left(h^{-1}\left(\frac{\epsilon}{2}\right)+R-R\right)=\frac{\epsilon}{2}.$$\end{proof}

\begin{proof}[Proof of Proposition \ref{p: epsilon-R estimate}.]
We show that a lower bound of
$\Lambda(\epsilon,R)$ on the length of an edge $e$
guarantees an upper bound on the movement of the midpoint $m$ of $e$. 

Let $F_s$ denote the fixed-point set of the involution $s$.
Let $x_i$ denote the orthogonal projection of $v_i$ to the fixed-point 
set $F_s$. Let $P:e\to [x_1x_2]$ denote the orthogonal projection
from $e$ to the geodesic segment $[x_1x_2]$, and $m$ be
the midpoint of $e$.  Then $d(m,P(m))\leq d(v_i,x_i)\leq \frac{R}{2}$
by Corollary \ref{c: convexity corollary}.

Suppose that $x_1$ and $x_2$ are distinct points.
Set $x=P(m)$.  Either $\angle xmv_2\geq \frac{\pi}{2}$ 
or $\angle xmv_1\geq \frac{\pi}{2}$, since they are complementary angles.  
Without loss of generality,
assume that $\angle xmv_2 \geq \frac{\pi}{2}.$  Then 
$\angle xs(m)s(v_2)\geq \frac{\pi}{2}$, and the quadrilaterals
$[xx_2v_2m], [xx_2s(v_2)s(m)]$ satisfy the conditions of 
Lemma \ref{l: kap quad lemma}.

Let $g=h^{-1}(\frac{\epsilon}{2})+R$ as in 
Corollary \ref{l: epsilon quad corollary}.   We assume
that the length of $e=[v_1v_2]$ is at least $2g$, so
$d(m,v_2)=d(s(m),s(v_2))\geq g$. Lemma \ref{l: epsilon quad corollary}
shows that $d(x,m)=d(x,s(m))\leq \frac{\epsilon}{2}$.
We conclude that $$d(m,s(m))\leq d(x,m)+d(x,s(m))\leq \epsilon.$$
The case when $x_1$ and $x_2$ coincide follows by continuity.

We have shown that when the length of $e$ is at least
     $$2h^{-1}\left(\frac{\epsilon}{2}\right)+2R,$$
the midpoint of $e$ is moved no more than $\epsilon$.

To complete the proof of the proposition, note that $h^{-1}(x)\leq \frac{1}{x}$.
Hence it is sufficient to take the length of the edge $e$ to be at least 
$\frac{4}{\epsilon}+2R$ as desired.
\end{proof}

\section{The quasi-convex hull is approximately a tree}
\label{s: qc hull}

Here, we show that the ``quasi-convex hull'' of a finite subset of $\H^n$
is quasi-isometric to the Gromov approximating tree for that subset.

Suppose $X$ is a finite subset of $\H^n$.  We define its quasi-convex hull $Q(X)$
as the union of geodesic segments between pairs of points of $X$:
   $$Q(X)=\bigcup_{x,y\in X} [xy]\subset \H^n.$$
We refer to the segments comprising $Q(X)$ as edges of $Q(X)$.

\begin{defn}
\label{d: q-i}
Recall that two spaces $X$ and $Y$ are $(L,A)$-{\em quasi-isometric} 
if, for a given $L>0, A\geq 0$, there is map $f: X\to Y$ such that the following are true:
\begin{enumerate}
  \item The map $f$ satisfies 
        $$-A + \frac{1}{L}d_X(x_1,x_2)\leq d_Y(f(x_1),f(x_2))\leq Ld_X(x_1,x_2)+A$$
        for all $x_1,x_2\in X$.
  \item There is a map $\overline{f}: Y\to X$ such that 
        $$-A + \frac{1}{L}d_Y(y_1,y_2)\leq d_X(\overline{f}(y_1),\overline{f}(y_2))
               \leq Ld_Y(y_1,y_2)+A$$
        for all $y_1,y_2\in Y$.
  \item The maps $\overline{f}$ and $f$ satisfy $d_X(x,\overline{f}{f}(x))\leq A$
         and $d_Y(y,f\overline{f}(y))\leq A$.
\end{enumerate}
If there is such an $f$, we call it an {\em $(L,A)$-quasi-isometry.}
We say a map $f:X \to Y$ is a {\em quasi-isometric embedding}
if it satisfies Property 1, but there is not necessarily a map $\overline{f}:Y\to X$
that satisfies Properties 2 and 3.
\end{defn}

The main result of this section is:
\begin{prop}
\label{p: Q and T are q-i}
For any finite set $X\subset \H^n$, there is a finite metric tree $T$
which is $(1,A)$-quasi-isometric to $Q(X)$. 
We may take $A=(20c+12)\ln 3$ for any $c>0$ such that 
$|X|\leq 2^c+2$.   Thus $A$ depends only on the cardinality of $X$.  
(We build the quasi-isometry $P:Q(X)\to T$ in 
 Lemma \ref{l: refinement of a q-i relation}.)
\end{prop}
 
The proof of Proposition \ref{p: Q and T are q-i} uses the quasi-isometry of a hyperbolic triangle and its
comparison tripod (Definition \ref{d: comparison tripod}).  We construct the 
quasi-isometry by first
considering the union of maps from individual triangles in $Q(X)$ to tripods 
in $T$.  The union forms a relation, and we refine
the relation into the map $P$.  We take $T$ to be the Gromov approximating tree (see Definition \ref{d: GAT}).

\begin{defn}
\label{d: comparison tripod}
Let $x_1,x_2,x_3 \in \H^n$, and let $\Delta$ denote the triangle $[x_1x_2x_3]$.  
The unique tripod $\tau_\Delta$ with endpoints $a_1,a_2,a_3$ such
that $d_\H(x_i,x_j)=d_T(a_i,a_j)$ for all $i,j$
called the {\em comparison tripod} of $\Delta$.

Let $\chi_\Delta$ be the unique map sending $x_i\in X$ to $a_i\in \tau_\Delta$ 
for each $i$ and which restricts to an isometry along each edge $[x_ix_j]\in X$. 
\end{defn}

\begin{defn}
\label{d: d-thin}
Fix $\delta \geq 0$.
We say that $\Delta$ is {\em $\delta$-thin} if the preimage $y,z\in \chi_\Delta^{-1}(x)$
satisfies $d_X(y,z)\leq \delta$ for all $x\in X$. 
\end{defn}

\begin{prop}[{\cite[4.3]{CDP}}]
\label{p: Hn is delta-hyperbolic}
All triangles in real hyperbolic $n$-space $\H^n$ are delta-hyperbolic
for $\delta=\ln{3}$.  
\end{prop}
Hence we say that $\H^n$ is $\delta$-hyperbolic for $\delta=\ln 3$.

\begin{defn}
\label{d: GAT}
Let $X$ be a finite subset of $\H^n$
with cardinality $|X|\leq 2^c+ 2$ for $c>0$.
We say that a pair $(T,p:X\to T)$ 
is a {\em Gromov approximating tree} if $T$ is a finite metric tree and $p:X\to T$
has the following properties:
\begin{enumerate}
  \item The map $p: X \to T$ sends points in $X$ to vertices of $T$, and 
    $$\partial T \subset p(X),$$ where $\partial T$ denotes the leaves of $T$ (the valence-one vertices).
  \item The distance between pairs of points is quasi-preserved and $p$ 
   does not increase distance:
    $$d_\H(x_1,x_2)-2c\delta \leq d_T(p(x_1),p(x_2))$$
  for all $x_1,x_2\in X$.  
\end{enumerate}
\end{defn}
We sometimes abbreviate {\em Gromov approximating tree} to {\em approximating tree}.

\begin{thm}[{\cite[Section 2.1]{GD}}]
\label{t: GAT}
Let $X$ be a finite subset of $\H^n$ and
cardinality $|X|\leq 2^c+2$ for $c>0$.  Then there exists
a pair $(T,p)$ with the properties described in 
Definition \ref{d: GAT}.
\end{thm}

\begin{rmk}
Gromov's construction applies to finite $\delta$-hyperbolic spaces
and finite sets of rays in $\delta$-hyperbolic spaces.  
\end{rmk}

\subsection{Triangles in $Q(X)$}
\label{s: triangles in Q}
Let $X$ be a finite subset of $\H^n$
with cardinality $|X|\leq 2^c+2$ for $c>0$, and let $(T,p)$
be a Gromov approximating tree for $X$.
Such a tree always exists by Theorem \ref{t: GAT}.
We show below that triangles in $Q(X)$ are uniformly quasi-isometric
to triangles in $T$
by extending $p:X\to T$ to 
triangles in $Q(X)$.

Let $[xy]_T$ denote
the geodesic segment between two points $x,y$ in a tree $T$, and
Let $[xyz]_T$ denote the tripod in $T$ with leaves $x$, $y$, $z$.
We suppress subscripts where there is no ambiguity.

Given $x_1,x_2,x_3 \in X$, let $\Delta$ be the triangle 
$[x_1x_2x_3]\subset \H^n$.  
Let $$T_\Delta=[p(x_1)p(x_2)p(x_3)] \subset T$$ be the tripod
in $T$ with leaves $p(x_1)$, $p(x_2)$, $p(x_3)$.   Let 
$$o=[p(x_1)p(x_2)]\cap[p(x_2)p(x_3)]\cap[p(x_3)p(x_1)]$$
be the branch point  of $T_\Delta$.  
Let $\tau_\Delta=[a_1a_2a_3]$ 
be the comparison
tripod (Definition \ref{d: comparison tripod}) of $\Delta$, and let 
$o_\tau=[a_1a_2]\cap[a_2a_3]\cap[a_3a_1]$ be the branch point of $\tau_\Delta$.
Defining $\chi_\Delta:\Delta\to \tau_\Delta$  as in Definition
\ref{d: comparison tripod},
we call the points in $\chi_\Delta^{-1}(o_\tau)$ 
the {\em internal points}. Note that there is one internal
point for each side $[x_ix_j]$. We label them $o_{12}, o_{23}, o_{13}$,
where $o_{ij}$ is contained in the side $[x_ix_j]$.

We extend $p|_{\{x_1,x_2,x_3\}}$ to the unique map $p_\Delta:\Delta\to T_\Delta$ that 
 \begin{enumerate}
   \item sends $x_i$ to $p(x_i)$ for each $i$
   \item sends the point $o_{ij}$ to $o$ for each $i,j$
   \item maps the segment $[x_io_{ij}]\subset \Delta$ to the segment
      $[p(x_i)o]\subset T_\Delta$ as a dilation:
        $$d_T(p(x_i),p_\Delta(x))=\frac{d_T(p(x_i),o)}{d_\H(x_i,o_{ij})}\cdot d(x_i,x).$$
 \end{enumerate}

\begin{prop}
\label{p: triangles in Q}
The map $p_\Delta:\Delta \to T_\Delta$ is a
$(1,4c\delta+2\delta)$-quasi-isometry, where $\Delta$ is given
the subspace metric from $\Delta\subset \H^n$.
\end{prop}

\begin{proof}
The map $p_\Delta$ restricts to a $(1,4c\delta)$-quasi-isometric embedding
along an edge $[x_ix_j]$ in $\Delta$.  To show this, we express the distance $d(p(x_i),o)$ 
as 
$$   d_T(p(x_i),o)=\frac{1}{2}(d_T(p(x_i),p(x_{i+1}))+d_T(p(x_i),p(x_{i-1}))-d_T(p(x_{i+1}),p(x_{i-1}))),$$
and combine this with the equations
$$   d_\H(x_i,o_{ij})=\frac{1}{2}(d_\H(x_i,x_{i+1})+d_\H(x_i,x_{i-1})-d_\H(x_{i+1},x_{i-1}))$$
and
$$ d(x_i,x_j)-2c\delta \leq d(p(x_i),p(x_j)) \leq d(x_i,x_j).$$

To complete the proof, we consider any $x,y\in\Delta$.  In a 
$\delta$-thin triangle, it is always possible
to find $x'\in\Delta$ and $y\in \Delta$ so that $x'$, $y'$ lie along a common
side and $d_\H(x,x')\leq \delta$ and  $d_\H(y,y')\leq\delta$.  
Hence the map $p_\Delta$ is a 
$(1, 4c\delta+2\delta)$-quasi-isometric embedding.   Since $p_\Delta$
is surjective, it is a quasi-isometry.
\end{proof}
\subsection{Constructing the quasi-isometry between $Q(X)$ and $T$}
\label{s: constructing q-i for Q and T}

To construct the map needed for Proposition \ref{p: Q and T are q-i}, 
we use the triangle maps constructed in Section \ref{s: triangles in Q} 
to build a relation $P':Q(X)\to T$, and then refine the relation
into the desired map.

We define the
relation $$P':Q(X)\to T$$
as follows: given $x\in Q(X)$, the image of $x$ is the set
    $$\{ p_\Delta(x) \in T \mid 
           x\in \Delta \subset Q(X)\}$$
where $p_\Delta$ is defined as in Section \ref{s: triangles in Q}.
Note that for each $x\in X$, the point $p(x)$ is contained in the image set $P'(x)$.

\begin{prop}
\label{p: relation P'}
The relation $P':Q(X)\to T$ satisfies the following three properties:
 \begin{enumerate}
 \item There exists a constant $A'$ that uniformly bounds 
    the diameter of $P'(x)$ for all $x$.
 \item There exists $A$ so that for each pair of points $x_1, x_2\in Q(X)$, 
  we can find $y_1\in P'(x_1)$ and $y_2\in P'(x_2)$ so that
  $$d(x_1,x_2)-A\leq d(y_1,y_2)\leq d(x_1,x_2)+A.$$
We may take $A=4c\delta+4\delta$ and $A'=8c\delta+4\delta$,
so $A$ and $A'$ depend only on the cardinality of $X$.
  \item The relation $P'$ is surjective.
 \end{enumerate}
\end{prop}

If we can show Proposition \ref{p: relation P'}, when we can use
the following to complete the proof of Proposition \ref{p: Q and T are q-i}.

\begin{lemma}
\label{l: refinement of a q-i relation}
Let $e$ be an edge of $Q(X)$.
Suppose a relation $P':Q(X)\to T$ satisfies the conditions of 
Proposition \ref{p: relation P'}.
Then there is a quasi-isometry $P:Q(X)\to T$ so that $P(x)\in P'(x)$ 
for all $x$, the vertices of $T$ are contained in the image of $P$, 
and $P$ is continuous on $e$.
\end{lemma}
\begin{proof}[Proof of Lemma \ref{l: refinement of a q-i relation}]
We construct a map $P$, then use the assumed conditions of 
Proposition \ref{p: relation P'} to show that it satisfies
the desired quasi-isometric inequalities.

First note that by construction of $p_\Delta$, 
when $x\in X$, we have $P'(x)=p(x)$ (see Section
\ref{s: triangles in Q}).  So, to construct $P$ from $P'$,
we pick images for points along edges $[x_1x_2]\subset Q(X)$,
where $x_1,x_2\in X$.  Let the vertices of $X$ be $\{x_1,\dots,x_k\}$.
Since $X$ is a finite set,
there are finitely many edges in $Q(X)$; we order the edges 
$e_1, e_2, e_3$, \dots, etc. Without loss of generality, suppose
that $e=e_1$ and $x_1$ is a boundary vertex of $e_1$. We may choose the 
edges so that $e_i=[x_1x_i]$ for $2\leq i\leq k$.

  For each edge $e_i$ in $Q(X)$, we
pick a triangle $\Delta_i$ containing $e_i$.  Then, for each
point $x\in e_i\setminus \{e_1\cup \dots \cup e_{i-1}\}$, we
set $P(x)=p_{\Delta_i}(x)$. One can check that the image of
the map $P:Q(X)\to T$ thus
defined contains all vertices of $T$.  By construction,
$P$ is continuous on $e$.

To show that $P$ is a quasi-isometry, let $x_1, x_2 \in Q(X)$. 
Combining Property 1 and Property 2 of Proposition \ref{p: relation P'}, 
we have
   $$d(x_1,x_2)-(A+2A')\leq d(P(x_1),P(x_2))\leq d(x_1,x_2)+(A+2A').$$
Hence $P$ is a $(1,A+2A')$-quasi-isometric embedding. 
Property 3 of Proposition \ref{p: relation P'} allows us to construct
a quasi-inverse map; to show that the quasi-inverse inequalities are
satisfied, we use Properties 1 and 2 of \ref{p: relation P'}.
\end{proof}

\begin{proof}[Proof of Proposition \ref{p: relation P'}]
To show Property (1), let $x\in Q(X)$.  Let $\Delta_1, \Delta_2\subset Q(X)$
be two triangles containing $x$, and suppose that 
$[x_1x_2]\subset \Delta_1,\Delta_2$ is the edge of $Q(X)$ containing $x$.
By Proposition \ref{p: triangles in Q}, 
   $$d_\H(x,x_1)-4c\delta-2\delta\leq d_T(p_{\Delta_j}(x),x_1)\leq 
       d_\H(x,x_1)+4c\delta+2\delta.$$ 
Since $p_{\Delta_1}(x)$ and $p_{\Delta_2}(x)$ lie along the geodesic
$[p(x_1)p(x_2)]$, it follows that 
$$d(p_{\Delta_1}(x),p_{\Delta_2}(x))\leq 8c\delta+4\delta.$$
Hence we may take $A'=8c\delta+4\delta$.

To show Property (2), let $x_1,x_2\in Q(X)$. If $x_1$ and $x_2$ lie
along edges that share a boundary vertex $x\in X$, then the points $x_1$ and 
$x_2$ lie in a triangle $\Delta\subset Q(X)$.  In this case, Property (2) follows
from Proposition \ref{p: triangles in Q}.   If $x_1$ and $x_2$ lie
on disjoint edges in $Q(X)$, then they lie on opposite sides of a quadrilateral
in $Q(X)$.  In this case, it is possible to find two points $x_1'$ and $x_2'$
in a common triangle
so that $d_\H(x_i, x_i')\leq \delta$.  Let $\Delta'$
be a triangle containing $x_1'$ and $x_2'$. Then 
$$d(x_1,x_2)-4c\delta-4\delta\leq d(p_{\Delta'}(x_1),p_{\Delta'}(x_2))
\leq d(x_1, x_2)+4c\delta+4\delta.$$

To show Property (3), let $x_1,x_2\in X$ and $\Delta\subset Q(X)$ 
be a triangle in $Q(X)$ containing $[x_1x_2]$.  
Then $P'([x_1x_2])$ is the geodesic segment $[p(x_1)p(x_2)]_T$ of $T$.  
Since $p(X)$ contains all the leaves of $T$, the image $P'(Q(X))$ 
covers all geodesic segments between leaves of $T$.  Hence $P'$ is
surjective.
\end{proof}

\begin{proof}[Proof of Proposition \ref{p: Q and T are q-i}]
Let $P: Q(X)\to T$ 
be the map constructed in Lemma \ref{l: refinement of a q-i relation}.
Then $P$ is an extension of $p$.
It follows from Proposition \ref{p: relation P'} and 
Lemma \ref{l: refinement of a q-i relation} that $P$ is
a $(1,20c\delta+12\delta)$-quasi-isometry between $Q(X)$ and $P(Q(X))$.
\end{proof}

What we have shown can be summarized as:
\begin{prop}
\label{p: rephrased Q and T are q-i}
Let $e$ be an edge in $Q(X)$.  Then one can find a map $P_e: Q(X)\to T$
that is continuous on $e$, and with the property that
given $x\in Q(X)$, there is a triangle $\Delta_x \subset Q(X)$
where $x\in \Delta_x$ and $P_e(x)=p_{\Delta_x}(x)$. 
The map $P_e$ is an extension of $p$, all vertices of $T$ 
are contained in the image of $P_e$, and $P_e$ is
a $(1, 20c\delta+12\delta)$-quasi-isometry between $Q(X)$ and $T$.
\end{prop}

\section{The shadow of an approximating tree in $\H^n$}
\label{s: shadow}

The purpose of this section is to define a projection of $T$ into $\H^n$, called
the ``shadow'' $T_\sh$ of $T$.
The shadow is a collection of geodesic segments in $\H^n$, and contains $X$.  

To set up the definition of the shadow, we define 
a subset $\overline{X}\subset Q(X)$ as follows: let
$V(T)$ denote the vertices of $T$,
and let
$P:Q(X)\to T$ be a map satisfying the conditions of 
Proposition \ref{p: rephrased Q and T are q-i}.
Then, given $y\in V(T)$,
we assign to $y$ a point $x\in Q(X)$ so that $P(x)=y$, with
the requirement that if $y\in P(X)$, then the chosen $x$ is
an element of $X$. It is always possible to arrange the assignment 
so that no two points in $V(T)$ are assigned to the same $x$.  This follows
by construction of the map $P$ and the
definition of the Gromov approximating tree.
Denote the set of points chosen as $\overline{X}$.  The assignment
gives a bijection, which we denote as $q_V:V(T)\to \overline{X}$. 
We extend $q_V$ to a map $q$, which will allow
us to define the ``shadow'' of $T$.

\begin{prop}
\label{p: Tsh and T are q-i}
Let 
$q:T\to \H^n$ be the extension of $q_V:V(T)\to \overline{X}$ which is the unique map sending the edge $[y_1y_2]\subset T$
to $[q_V(y_1)q_V(y_2)]\subset \H^n$ via dilation: given 
$y\in [y_1y_2]$, we have
   $$d_\H(q_V(y_1),q_V(y))
      =\frac{d_\H(q_V(y_1),q_V(y_2))}{d_T(y_1,y_2)}\cdot d_T(y_1,y).$$
Along each edge in $T$, the map $q$ restricts to a $(1, 20c\delta+12\delta)$-quasi-isometry.
\end{prop}

\begin{proof}
If $y_1,y_2$ are vertices of $T$, then $P(q_V(y_i))=y_i$,
for $i=1,2$. Hence we may apply
Proposition \ref{p: rephrased Q and T are q-i} 
to the map $P: Q(X)\to T$ restricted to the points $y_1$ and $y_2$,
yielding the desired quasi-isometric inequality.
\end{proof}

\begin{rmk}
The map $q$ is in fact a $(1, |X|(20c\delta+12\delta))$-quasi-isometry between $T$ and $q(T)$.
\end{rmk}

\begin{defn}
We call a point $x=q(y)\in q(T)$ the {\em shadow} of $y\in T$.
\end{defn}

To define the ``shadow'' of the tree $T$, we first
observe that the image $q(T)$ contains the set $\overline{X}$.
Recall the map $p:X\to T$ (Definition \ref{d: GAT}).  When $p$ is not
injective, then $\overline{X}$ does not contain the original set $X$
generating the Gromov approximating tree.   
However, 
if $x\in X\setminus\overline{X}$, there is a unique element $z$ of $\overline{X}$ such that 
$p(x)=p(z)$.
 
For ease of exposition in later sections, we define the shadow of
$T$ to be a connected union of geodesic segments in $\H^n$ containing 
$q(T)$ and $X$.

\begin{defn}
\label{d: Tsh}
We define the {\em shadow} of $T$, denoted $T_\sh$, as the union of 
the image $q(T)$ and segments $[xz]\subset \H^n$ chosen as follows:
if $x\in X$ is not an element of $\overline{X}$, then the segment 
$[xz]$ is included in $T_\sh$, where $z$ is the unique element of $\overline{X}$
such that $p(x)=p(z)$.

If $x_1$ and $x_2$ are points in $T_\sh$, we define $d_\sh(x_1,x_2)$ 
to be the distance of a shortest path in $T_\sh$ from $x_1$ to $x_2$.
\end{defn}

The shadow $T_\sh$ is a collection of geodesic segments $[xy]\subset \H^n$
whose combinatorics mimic those of $T$.  

We let $[xy]_\sh$ denote the following union of segments 
in $T_\sh$: let $z_1, z_2$ be
the unique elements of $\overline{X}$ such that $p(x)=p(z_1)$ and $p(y)=p(z_2)$.
Then we define $[xy]_\sh$ as the concatenation of the segments $[xz_1]$,
$q([p(x)p(y)])$, and $[yz_2]$.

\section{Labelling Systems}
\label{s: labelling systems information}

In this section, we develop 
a purely combinatorial framework for working with 
special splittings, called a {\em labelling system}. 
Labelling systems are a collection
of labels for vertices of a tree; in Section \ref{s: application},
we will use them to produce special splittings from edges of 
an approximating tree.  Nontrivial splittings 
correspond to {\em useful} edges; trivial splittings correspond to
{\em useless edges}.  Whether an edge is useful or useless can
determined combinatorially.

The main result for this section is Proposition 
\ref{prop: useful subtree}, which says that the union of useful
edges is connected.  The crux of the proof of 
Proposition \ref{prop: useful subtree} is 
Proposition \ref{p: useless means full}, which is essentially
the Topological Helly Theorem, applied to the context
of labelling systems.
 
\subsection{Labelling systems}
\label{s: labelling systems}

Let $T$ be a finite simplicial tree.  Recall that a valence-one vertex is a {\it leaf} and the
set of leaves is $\partial T$.
Recall that $[ab]\subset T$ denotes the minimum length path between vertices $a$ and $b$ 
of $T$.

\begin{defn}
\label{d: labelling}
A {\it labelling} of $T$ is a relation 
  $\Lab: \vx(T)\to \{1, \dots, N\}.$
In particular, a vertex may have zero or more than one labels.
\end{defn}

\begin{defn}
\label{d: labelling system}
Let $\Lab(v)$ be the set of labels assigned to a vertex $v$.
We say that a relation
$$\Lab: \vx(T) \to \{1,\ldots, N\}$$ 
is {\em labelling system} if it
satisfies the following properties:

{\bf Property A (connectedness)}.  Let $a$ and $b$ be vertices of $T$, 
and let $x\in \vx(T)$ be a vertex
contained in the path $[ab]\subset T.$   Then 
$\Lab(a)\cap\Lab(b)\subset \Lab(x).$

{\bf Property B (surjectivity)}.  The {\em full set} of indices 
 $\{1,\ldots, N\}$  is contained in $$\bigcup_{x} \Lab(x),$$
 where the union is taken over all vertices of $T$.

\end{defn}

The labelling system used in Section
\ref{s: application} is constructed from an
existing labelling as follows.

\begin{defn}
\label{d: old-new labelling system} 
Let $(T,\Lab)$ be a labelled tree, and let $V$ denote 
the vertices of $T$.    We define a labelling 
     $\newLab: V=\vx(T)\to \{1,\dots, N\}$
as follows.  Suppose $x$ is a vertex of $T$.   Let $Z(x)$ be the
set of minimum-length paths in $T$ passing through $x$, so
     $Z(x)=\{ [ab] \mid x\in [ab]\}.$
We set
     $$\newLab(x)=\bigcup_{[ab]\in Z(x)} (\Lab(a)\cap \Lab(b)),$$
so if $x$ lies in the path $[ab]$ and $a$ and $b$ have a common
label $i$, then $i\in\Lab(x)$.
We call $\newLab$ the {\em canonical extension} of $\Lab$.
\end{defn}
It follows that $\Lab(x)\subset \newLab(x)$ for all vertices $x$ in $T$.

Using standard techniques for working with paths in trees, one may
verify the following two lemmas.

\begin{lemma}
\label{l: new implies old}
Let $T$ be a tree, $\Lab$ be a labelling of $T$, and $\newLab$ be the canonical
extension of $T$.
Suppose $v$ is a vertex in $[ab]$ and $i\in \newLab(a)\cap\newLab(b)$. 
Then then there exist $a',b'$ such that $v\in[a'b']$ and $i\in\Lab(a')\cap\Lab(b')$.
\end{lemma}

\begin{lemma}
\label{lemma: combinatorial label}
Let $\Lab$ be a labelling on $T$, and let $\newLab$ is the canonical extension
of $\Lab$.  Then $\newLab$ is connected (Property A from Definition
\ref{d: labelling system}.
\end{lemma}

It is an immediate consequence of Lemma \ref{lemma: combinatorial label} that:
\begin{lemma}
\label{l: new is surjective}
If $\Lab$ is surjective (Property B from Definition \ref{d: labelling system}),
then $\newLab$ is a labelling system.
\end{lemma}

We use the Lemmas \ref{lemma: combinatorial label}-\ref{l: new is surjective}
in Section \ref{s: geometry and combinatorics}, when we relate the geometry
of approximating trees to the combinatorics of labellings.

\subsection{Useless and useful edges}
\label{s: useful subtree}
Suppose that $(T,\Lab)$ is a finite simplicial labelled tree, and that $\Lab$
is a labelling system (Definition \ref{d: labelling system}).

The removal of any open edge $e\subset T$ separates the tree
into two closed connected components.  For the sake of bookkeeping,
let us orient the edge.  We call $T^+(e)$ the component toward which
$e$ is oriented and we call the remaining component 
$T^-(e)$, as illustrated in Figure \ref{fig: plus and minus subtrees}.

\begin{figure}
\includegraphics{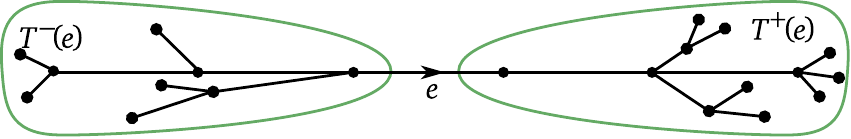}
\caption{\label{fig: plus and minus subtrees} The subtrees $T^+(e)$ and $T^-(e)$.}
\end{figure}

\begin{defn}
\label{d: useful edge}
We say that an edge is {\em useless} if
   $$\bigcup_{v\in T^+} \Lab(v) \quad\; \textrm{ or } \quad \bigcup_{v\in T^-} \Lab(v)$$  
contains the full index set. An edge is {\em useful}
if it is not useless.  
\end{defn}

\begin{prop}
\label{p: useless means full}
Every edge of $T$ is useless if and only if there exists
a ``full vertex'', i.e., a vertex $z$ such that $\Lab(z)=\{1,\dots,N\}.$
\end{prop}

\begin{proof}
Suppose that $z\in \vx(T)$ is full.  Let $e$ be an edge of $T$.
Then $z$ is contained in either $T^+(e)$ or $T^-(e)$, so $e$ is
useless. Hence all edges are useless.

The other direction follows from the Topological Helly Theorem in
\cite[Lemma $A_m$]{Deb}. When working with a finite collection $\{T_i\}$ of contractible sets 
in a contractible space $T$,
the Topological Helly Theorem states that if the space $T$ has
covering dimension 1 and the pairwise intersection
$T_i\cap T_j$ is nonempty and connected for all $i\neq j$, 
then the intersection $\bigcap T_i$ is nonempty.

In our case, let $T_i$ be the subtree of the tree $T$ spanning all vertices
labelled by $i$.   By surjectivity of $\Lab$ (Definition \ref{d: labelling system},
Property B), 
each $T_i$ is nonempty.   
By construction, each $T_i$  is contractible. 

To show that $T_i$ and $T_j$ intersect, suppose
by contradiction that they do not.  
Then there exists an edge $e$ that separates $T_i$ from $T_j$, i.e., an edge $e$
such that $T_i\subset T^+(e)$ and $T_j\subset T^-(e)$.  This contradicts
the assumption that all edges are useless.

We have shown that $\bigcap T_i$ is nonempty.  Because each $T_i$ is a finite
simplicial tree, and there are only a finite number of $T_i$, their
intersection contains at least one vertex.  Hence there exists 
a vertex labelled by all $i$ in $\{1,\ldots, N\}$.
\end{proof}

\begin{prop}
\label{prop: useful subtree}
The union of useful edges of $T$ forms a subtree.
\end{prop}

\begin{figure}
\includegraphics{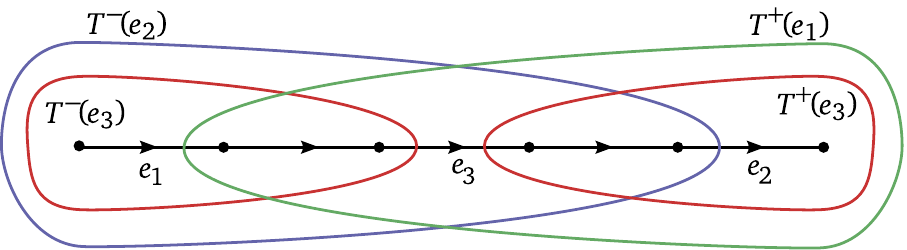}
\caption{\label{fig: useful edges} Useless edges cannot separate useful edges.}
\end{figure}

\begin{proof}
Let $e_1$ and $e_2$ be useful edges.   Let $e_3$ be an
edge contained in the unique geodesic path between $e_1$ and $e_2$ 
and orient the path so it flows from $e_1$ to $e_2$
(see Figure \ref{fig: useful edges}). 
We show that $e_3$ is useful.

Since $e_1$ and $e_2$ are useful edges, neither 
$T^+(e_1)$ nor $T^-(e_2)$ contain all labels.  By way of contradiction, 
suppose that $e_3$ is a useless.  Then the vertices of
either $T^+(e_3)$ or $T^-(e_3)$ contain all the labels.  However, 
because $e_3$ lies between $e_1$ and $e_2$, we have
$T^+(e_3)\subset T^+(e_1)$ and $T^-(e_3)\subset T^-(e_2)$.
This means that either 
  $$\bigcup_{v\in T^+(e_3)} \Lab(v)=S \subset 
     \bigcup_{v\in T^+(e_1)} \subsetneq S$$
     or
   $$\bigcup_{v\in T^-(e_3)} \Lab(v)=S \subset 
     \bigcup_{v\in T^-(e_2)} \subsetneq S,
  $$
giving a contradiction.
\end{proof}

Recall that our ultimate aim is to associate edges in a Gromov
approximating tree to splittings. 
For the proof of the main result, we are interested in nontrivial splittings.  
As we show in Section \ref{s: special splittings from Thyp}, an edge 
produces a nontrivial splitting when it is {\it useful} in the sense
of Definition \ref{d: useful edge}.
For this reason, we let $\Tspl$ denote the subtree formed by useful edges.

\begin{prop}
\label{p: Labspl}
Let $\Labspl$ be the restriction of $\Lab$ to $\Tspl$. For nonempty $\Tspl$, the
relation $\Labspl$
is surjective (Definition \ref{d: labelling system}, Property B).
\end{prop}

\begin{proof}  
We assume
that $\Tspl$ is nonempty.  To show surjectivity of $\Lab_\spl$, we 
construct a tree $T'$ 
by collapsing the useful subtree to a point: $T' = T/\Tspl.$
Let $\rho_\spl:T\to T'$ be the quotient map that induces the identification.

The image $v_\spl=\rho_\spl(\Tspl)$ is a vertex of $T'$.
If $w$ is a vertex of $T'$ other than $v_\spl$, it lifts to a unique vertex in $T\setminus\Tspl$.
Hence we define  $\Lab': \vx(T')\to \{1,\dots,N\}$
to send $v_\spl$ to $\bigcup_{v\in\Tspl} \Lab(v)$
and other vertices $w$ to $\Lab(\rho_\spl^{-1}(w))$.

Since $\Lab$ is a labelling system, so is $\Lab'$.   Furthermore,
every edge of $T'$ is useless by construction.
By Proposition \ref{p: useless means full}, the tree $T'$ has
a full vertex $x$.  

We claim that $x=v_\spl$.   By contradiction, suppose that is it not.  
Then $$\Lab(\rho_\spl^{-1}(x))=\{1,\dots,N\},$$ so the vertex 
$\rho_\spl^{-1}(x)$ of $T$ is full. By Proposition \ref{p: useless means full}, the
existence of a full vertex implies that 
all edges of $T$ are useless.  This is a contradiction, as $\Tspl$ 
is nonempty.  Hence $x=v_\spl$, and $\bigcup_{v\in\Tspl} \Lab(v)=\{1,\dots,N\}$.
We conclude $\Labspl$ is a surjective relation.
\end{proof}

\section{Bounds on the displacement function}
\label{s: bounds on B}

\subsection{The space of discrete and faithful representations and the displacement
function}

An isometric action $G\curvearrowright \H^n$ is equivalent to a representation 
$\rho: G\to \Isom(\H^n)$; the {\it representation
variety} of $G$-actions on $\H^n$ is defined as 
   $$R(G,n)=\Hom(G,\Isom(\H^n))=\{ \rho: G\to \Isom(\H^n) \}.$$
   
We define the adjoint action $ad\curvearrowright \Isom(\H^n)$ of $\Isom(\H^n)$
on itself via conjugation: $ad(h)\cdot h'=h^{-1}h'h$.  The adjoint action induces
an action $\Isom(\H^n)$ on $R(G,n)$; for $\rho\in R(G,n)$ and $h\in \Isom(\H^n)$,
the representation $h\cdot \rho$ sends $g$ to $h^{-1}\rho(g)h$.    The
space of conjugacy classes of representations in $R(G,n)$ is homeomorphic
to the quotient 
    $$\overline{R}(G,n)=R(G,n)/\Isom(\H^n),$$ 
where $\Isom(\H^n)$ acts on $R(G,n)$
by the action induced by $ad$.

Unfortunately, the above space is in general non-Hausdorff
\cite[Section 4.3, p. 57]{Kbook}.  So one instead considers
the Mumford quotient
  $$X(G,n)=\Hom(G,\Isom(\H^n))/\!/\Isom(\H^n),$$
which is an algebraic 
variety.  The space $X(G,n)$ is called
the {\it character variety}.  For more information
on this space, see \cite{Mo86}.  The series of work by Culler, Morgan,
and Shalen \cite{CS83}\cite{MS84}\cite{MS88a}\cite{MS88b}\cite{Mo86} examine the character
variety.  

We are interested in conjugacy classes of discrete
and faithful actions on $\H^n$.   
Let 
   $$\Hom_\df(G,n)\subset \Hom(G,\Isom(\H^n))$$
denote the space of discrete and faithful representations.
When $G$ does not contain any infinite nilpotent normal
subgroups (e.g., it is not small), then 
  $$\Hom_\df(G,n)/\Isom(\H^n)$$
is Hausdorff, and in particular, it is a subspace
of both $\overline{R}(G,n)$ and the character variety  $X(G,n)$
(see \cite[Chapter 8, p. 157]{Kbook}). 

\begin{defn}
Given a group $G$ and a dimension $n$, we define 
   $$\D(G,n)=\Hom_\df(G,n)/\Isom(\H^n),$$
and call this set the {\em deformation space} of $G$ 
into $\Isom(\H^n)$.
\end{defn}

\begin{defn}
Let $G$ be a finitely-presented group generated by $S$.  
Let $\rho:G\to \Isom(\H^n)$ be a representation, and let
$B_\rho(x)=\max_{s\in S} d(x,s(x)).$
The {\em displacement function} of a representation is defined as 
$B_\rho=\inf_{x\in \H^n} B_\rho(x).$  
We denote the supremum of displacement functions of representations
in a deformation space as $$B=\sup_{\rho\in \D(G,n)} B_\rho.$$
Given $[\rho], [\rho'] \in \D(G,n)$, we have $B_\rho = B_{\rho'}$  when $[\rho]=[\rho']$.
\end{defn}

In \cite{Be88}, Bestvina  observed that the Compactness Theorem is equivalent
to a uniform upper bound on the displacement function.  
As discussed in the introduction, the methods
used to prove the Compactness Theorem do not give estimates for
such a bound in general.  

\subsection{Application of combinatorial framework to special splittings}
\label{s: application}

Let $W$ be a Coxeter group with system $(W,S)$, and 
$\rho: W\curvearrowright \H^n$ 
be a discrete, faithful, and isometric action.  We associate
this data to an approximating tree $T$.
Below, we define the subset $X$ of $\H^n$ from which $T$ is constructed.

We assume that the Coxeter diagram $\Gamma(W,S)$ is connected; 
a disconnected Coxeter diagram corresponds to a splitting over
the trivial group, which is small.

Let $S=\{s_1,\dots,s_k\}$.
Consider the set $\mathcal S$ of pairs $\{s_i,s_j\}$ which generate finite
dihedral groups.  Being finite, these dihedral
groups have nonempty fixed-point sets in $\H^n$ (see \cite[Chapter II.6, Proposition 6.7]{BH}).  

Fix a representation $\rho:W\to \Isom(\H^n)$, and suppose it
is discrete and faithful.  Let $X$ be a set of representative
points from the fixed-point sets of pairs in $\mathcal S$;
hence $|X|\leq |\mathcal S|\leq \binom{k}{2}$.

\begin{rmk}
The space $X$ can have arbitrarily large diameter, even in the case $k=3$.
\end{rmk}

Let $(T,p)$ be a Gromov approximating tree for the set $X$, and 
recall the map $q:T_\sh \to T$.
Let $\stab_S(x)$ denote the set of elements in $S$ that fix $x$.

\begin{defn}
\label{d: newLab on T}
Let the labelling 
$\Lab:V(T)\to S$ 
send a vertex $v$ to the union of labels $\bigcup_{x\in p^{-1}(v)}\stab_S(x)$ when $q(v)\in X$
and to the empty set otherwise.
We define a {\em generator labelling}
denoted $\newLab:V(T)\to S$ as the canonical extension of $\Lab$.
\end{defn}

\begin{thm}
\label{t: Thyp labelling system}
The labelling $\newLab: V(T)\to S$ is a labelling system.
\end{thm}

\begin{proof}
We show that the map $\newLab$ is connected and surjective
(Definition \ref{d: labelling system}, Property A and Property B).  
The properties essentially
hold by construction.  

Property A follows from 
Lemma \ref{lemma: combinatorial label}.

To show Property B, note that each $\rho(s_i)$ has a nonempty fixed-point set because it is
an involution. The diagram for $(W,S)$ is connected, so there exists
an $s_{j\neq i}$  such that $\{s_i,s_j\}$ generate
a finite dihedral group, which has nonempty fixed point set.
For each $s_i$, there is a point $x_i\in X$ fixed by $s_i$.
Thus there is a point in $V(T)$ labelled by $s_i$, namely, the point $p(x_i)$. 
Since $p(X)\subset V(T)$,
it follows that the union $\bigcup_{v\in V(T)} \newLab(v)$
contains the full set $S=\{s_1,\dots,s_g\}$, so $\newLab$ is surjective.
\end{proof}

\subsection{Correspondence between the combinatorics of labelling systems 
and the geometry of actions}
\label{s: geometry and combinatorics}

Let $\Lab: V(T)\to S$ and $\newLab: V(T)\to S$ be the maps defined in Definition 
 \ref{d: newLab on T}. 
   
The combinatorics of the generator labelling system $\newLab$ correspond
to the geometry of the action: if $s_i$ labels a vertex 
$v\in T$, then we can bound the amount by which 
$\rho(s_i)$ displaces its shadow $q(v)$.  To make
this statement precise, we introduce the notion of an $R$-fixed point.

\begin{defn}
\label{d: R-fixed}
Let $R\geq 0$.
Fix an action $W\curvearrowright \H^n$.
We say a point $x$ is {\it $R$-fixed} 
by elements $w_1,\ldots,w_m$ of $W$ 
if
  $$d(x,w_i(x))\leq R \textrm{ for all } i\in \{1,\ldots,m\}.$$
\end{defn}

\begin{prop}
\label{p: s fixes point on q-geodesic}
Let $s\in S$. Suppose $u,w$ are vertices in $T$ and $x,y$ are vertices in $T_\sh$
so that $P(x)=u$, $P(y)=w$, and $s\in \Lab(u)\cap \Lab(w)$.
If $z\in [xy]_\sh$, then $z$ is $R$-fixed by $s$, where 
$R=2^8(|X|(20c\delta+12\delta)+4c\delta)$.
\end{prop}

As a consequence of Proposition \ref{p: s fixes point on q-geodesic}, we
obtain:

\begin{prop}
\label{p: old-new label system on Thyp}
Suppose that $v$ is a vertex of $T$ and $s\in \newLab(v)$.
Then $q(v)$ is $R$-fixed by $s$, where $R$
is defined as in Proposition \ref{p: s fixes point on q-geodesic}.
\end{prop}

\begin{proof}
If $s\in \newLab(v)$, then by Lemma \ref{l: new implies old}, there are vertices
$u,w\in T$ and $x,y\in T_\sh$ 
such that $v\in [uw]$,  
$s\in \Lab(u)\cap \Lab(w)$, and $p(x)=u, p(y)=w$. 
 Since $x$ and $y$
are fixed by $s$, it follows from Proposition
\ref{p: s fixes point on q-geodesic} that the vertex $q(v)$ is $R$-fixed by $s$.
\end{proof}

\begin{proof}[Proof of Proposition \ref{p: s fixes point on q-geodesic}.]

We use the following lemma.

\begin{lemma}[{\cite[Lemma 3.43, pp. 48-49]{Kbook}}]  
\label{l: q-i Hausdorff distance}
Let $\gamma=[xy]_\H$ be a geodesic and $\widehat{\gamma}$
be an $(1,A) $-quasi-geodesic path from $x$ to $y$.  Then 
$\widehat{\gamma}$ is contained in a regular $r$-neighbourhood 
$N_{r}(\gamma),$
where $r=2^7A.$
\end{lemma}

Let $\gamma=[xy]_\H$ and
$\widehat{\gamma}=[xy]_\sh$. We claim that $\widehat{\gamma}$ is a $|X|(20c\delta+12\delta)+4c\delta$
quasi-geodesic.
To show this, let $P_\gamma:Q(X)\to T$ be a map
sending $\gamma$ to $[uw]$ satisfying the conditions of
Lemma \ref{l: refinement of a q-i relation}.
The map $P_\gamma$ is distance decreasing 
on elements of $X$.
Let $z_1$ and $z_2$ be the unique points of $\overline{X}$ 
such that $p(x)=p(z_1)$ and $p(y)=p(z_2)$.
Let $\widehat{\gamma}'$ be the segment of $\widehat{\gamma}$
contained in the image of $q$; it is a sequence of edges of $T_\sh$
connecting $z_1$ to $z_2$.  
Recall that when restricted to each edge, the map $q$ is a
$(1,20c\delta+12\delta)$-quasi-isometry which is a homeomorphism.
Thus, we can construct a $(1,|X|(20c\delta+12\delta))$-quasi-isometry $\alpha:\widehat{\gamma}'\to [uw]$ by defining $\alpha$ as $q^{-1}$ along
edges. Then $q\circ P_\gamma$ sends $\gamma$ to $\widehat{\gamma}'$,
so the length of $\widehat{\gamma}'$ is at most $d_\H(x,y)+|X|(20c\delta+12\delta)$.

By construction of $T_\sh$ (Definition \ref{d: Tsh}) and 
the definition of a Gromov approximating tree (Definition \ref{d: GAT},
Property 2),
points in $T_\sh$ not contained in $q(T)$ are at most $2c\delta$
away from $q(T)$. Since the shortest path between $x$ and $y$ is at least as long as $d_\H(x,y)$,
it follows that
  $$d_\H(x,y)\leq d_\sh(x,y)\leq d_\H(x,y)+|X|(20c\delta+12\delta)+4c\delta.$$
Set $A=|X|(20c\delta+12\delta)+4c\delta$.
Let $a$ be a point on $\widehat{\gamma}$.
Let $b$ be the
nearest point on $\gamma$ to $a$. 
The element $s$ fixes $x$ and $y$, so $s$ fixes $\gamma$
and sends $[ab]_\H$ to 
$[s(a)b]_\H$. The distance $d(a,b)=d(s(a),b)$ is bounded
above by $r=2^7A$ as a consequence of 
 Lemma \ref{l: q-i Hausdorff distance}.
It follows that $d(a,s(a))$ is bounded above by $R=2^8A$.

We have shown that all points of $\widehat{\gamma}$ are $R$-fixed
by $s$.  
\end{proof}

\subsection{Special splittings produced by Gromov approximating trees}
\label{s: special splittings from Thyp}
An edge $e$ of the Gromov approximating tree $T$
determines a special splitting in the following way.  
\begin{prop}
\label{prop: definition of special splitting}
Define 
sets of generators
    $$S^+(e)=\bigcup_{v\in T^+(e)} \{s\in \Lab(v)\},$$
    $$S^-(e)=\bigcup_{v\in T^-(e)} \{s\in \Lab(v)\},$$
    $$S^*(e)=S^+(e)\cap S^-(e).$$   

The special splitting
$\genby{S^+(e)}\ast_{\genby{S^*(e)}}\genby{S^-(e)},$
where the amalgamation maps are induced by the inclusions $S^*(e)\hookrightarrow S^{\pm}(e)$,
yields a group isomorphic to $W$.  Furthermore,
the splitting is trivial if and only if $e$ is
useless, i.e., either $S^+(e)=S$ 
or $S^-(e)=S$.
\end{prop}

\begin{proof}
Let $\Gamma$ denote the Coxeter diagram for the system $(W,S)$.
We let $\Gamma^*\subset\Gamma$ denote the subgraph
spanned by $S^*(e)$; we define $\Gamma^+$ and 
$\Gamma^-$ similarly.

By Proposition 
\ref{prop: special splittings separate Gamma}, it suffices to
show that $\Gamma^*$ separates $\Gamma$ and that 
$S\subset S^+(e)\cup S^-(e)$.

We first show that $\Gamma^*$ separates $\Gamma$.
Let $\gamma$ be a path from $\Gamma^+\setminus\Gamma^*$
to $\Gamma^-\setminus\Gamma^*$.  By way of contradiction,
suppose that $\gamma$ does not pass through $\Gamma^*$. 
Then $\gamma$ contains
vertices $s_+\in S^+(e)\setminus S^*(e)$ 
and $s_-\in S^-(e)\setminus S^*(e)$ that are connected
by exactly one edge in $\Gamma$.  Hence the group 
$\genby{s_+, s_-}$ is a finite dihedral group, and 
$\{s_+, s_-\}$ is an element of $\mathcal{S}$.
So $\{s_+, s_-\}$ is contained in either $S^+(e)$ 
or $S^-(e)$.    Without loss of generality, suppose
that $\{s_+, s_-\}\subset S^+(e).$  Then $s_-\in S^+(e)\cap S^-(e)=S^*(e)$.
This contradicts the assumption that $s_-\in S^-(e)\setminus S^*(e)$.

We have shown that $\genby{S^+(e)}\ast_{S^*(e)}\genby{S^-(e)}$
is a splitting.  To complete the proof, it remains to show that
$W\cong \genby{S^+(e)}\ast_{S^*(e)}\genby{S^-(e)}$.  
It suffices to check that 
$S^+(e)\cup S^-(e)=S$.
This follows from the surjectivity of $\newLab$. 
Hence an edge $e$ determines the desired splitting. 
\end{proof}
The splitting over $e$ may be trivial.
In Section \ref{s: lower bound on displacement function}, 
we find a lower bound on the displacement
function of a representation to ensure the existence of a nontrivial
splitting.

\subsection{Lower bound on the displacement function to guarantee existence of
nontrivial splittings}
\label{s: lower bound on displacement function}

\begin{prop}
\label{p: useless tree} 
If every edge of $T$ determines a trivial splitting of $W$,
then $\Tsh$ contains a point that is $R$-fixed by all generators in $S$.  Hence if $B_\rho > R$,
there exists a nontrivial splitting of $W$.
\end{prop}

\begin{proof}
Suppose that every edge of $T$ determines a trivial splitting.
Then every edge of $T$ is useless.
By Proposition \ref{p: useless means full},
there exists a vertex $v\in T$ 
labelled by the full set $S=\{s_1,\dots,s_k\}$.  By Proposition
\ref{p: old-new label system on Thyp}, 
the point $q(v)\in V(T_\sh)$ is $R$-fixed by $S$.

If a point in $\H^n$ is $R$-fixed by $S$, then $B_\rho\leq R$.
This follows from
the definition of the displacement function and 
the definition of an $R$-fixed point.
This means that when $B_\rho > R$, there exists an edge $e$ 
which is useful and thus determines a nontrivial splitting
of $W$.
\end{proof}

By Proposition \ref{prop: useful subtree}, the union of edges of $T$ 
which determine non-trivial splittings of $W$ is connected.  
\begin{defn}
\label{defn: useful subtree}
Let $\Tspl$ denote
the subtree consisting of the union of useful edges and call this subtree
the \textit{useful subtree}.
\end{defn}
\begin{thm}[{{\bf Useful subtree theorem}}]
\label{t: useful subtree}
When $B_\rho >R$, the useful subtree $\Tspl$ is nonempty.
\end{thm}
\begin{proof}
By Proposition \ref{p: useless tree}, there exists a
useful edge $e$ of $T$.  
Hence $\Tspl$ is nonempty.
\end{proof}

\section{Small decompositions}
\label{s: small decompositions}

In this section, we prove the main result of the paper.

Let $\mu_n$ be the Kazhdan-Margulis constant for $\H^n$ (Theorem
\ref{t: K-M}).

\begin{ectcg}
Let $(W,S)$ be a Coxeter system and suppose $S$ has $k$ elements.
There exists a function $C_n(k)$ with the property that either $W$
has a small special nontrivial splitting or the displacement function
is bounded above by $C_n(k)$ for every discrete and faithful $W$-action
on $\H^n$.   We may take
   $$C_n(k)=R+2\binom{k}{2}\Lambda(\mu_n,R),$$
where $R=2^8(\binom{k}{2}(20c\delta+12\delta)+4c\delta)$,
and $\Lambda(\mu_n,R)=2\binom{k}{2}(\frac{4}{\mu_n}+2R)$ (as defined in Proposition \ref{p: epsilon-R estimate}).

As a function of $k$, the estimate $C_n(k)$ is of order $k^4$.
\end{ectcg}

\begin{defn}
We say that the {\it shadow length} of an edge $e$ in $T$
is the length of the quasi-geodesic segment $q(e)\subset \Tsh\subset \H^n$.
We let $|q(e)|$ denote the shadow length of $e$.
(See Proposition \ref{p: Tsh and T are q-i} for the definition of the map $q$.)
\end{defn}

We seek a condition on the shadow length of a useful edge $e$ 
that guarantees the existence of a small nontrivial splitting of $W$
(Definition \ref{d: small group}).

Let $\Lambda_n=\Lambda(\mu_n,R)$, where $R$ is defined as in Proposition
\ref{p: old-new label system on Thyp}.  Given an action $\rho: W\curvearrowright \H^n$ and
an edge $e$ of $T$, we define $S^*(e)$ as in Section 
\ref{s: special splittings from Thyp}.

\begin{lemma}
\label{c: e estimate}
If $|q(e)|\geq \Lambda_n,$ then  
then $S^*(e)$ generates a small group and
the special splitting associated with $e$ is small. 
\end{lemma}

\begin{proof}
Let $m$ be the midpoint of $q(e)$.  We show that $m$ is
$\mu_n$-fixed by $S^*(e)$.   According to Proposition  \ref{p: old-new label system on Thyp},
if $s\in S^*(e)$, then the shadow of each vertex in $e$ is $R$-fixed by $s$.  
Therefore, since 
$|q(e)|\geq \Lambda_n$, it follows from Proposition \ref{p: epsilon-R estimate}
that $d(m,s(m))\leq \mu_n$ 
for all  $s\in S^*(e)$.
The representation $\rho(\genby{S^*(e)})\subset \Isom(\H^n)$ is discrete,
so by the Kazhdan-Margulis Lemma (Theorem \ref{t: K-M}), the group $\rho(\genby{S^*(e)})$ is small.  
Since $\rho$ is a faithful representation, we conclude that $S^*(e)$ generates a small group.
\end{proof}

\begin{lemma}
\label{l: Tspl edge length bound}
There exists a function $C_n(k)$ with the following property.  
If the shadow length of every edge in $\Tspl$ 
is less than or equal to $\Lambda_n$, then each $x\in q(\Tspl)$ is $C_n(k)$-fixed by $S$.
Furthermore, we may take $$C_n(k)=R+2\binom{k}{2}\Lambda_n=2^8\left(\binom{k}{2}(20c\delta+12\delta)+4c\delta\right)+2\binom{k}{2} \Lambda_n.$$
\end{lemma}
\begin{proof}
If  $\Tspl$ is empty, then the lemma holds trivially, so
we work with nonempty $\Tspl$.

Let $x$ be a point in  $q(\Tspl)$, and let $s_i$ be an element of $S$.  
We analyse how far $s_i$
displaces $x$.  

By Proposition \ref{p: Labspl}, there is a vertex $v_i$
of $\Tspl$ with $s_i\in\newLab(v_i)$.
Since $|q(e)|\leq \Lambda_n$ for each edge $e$  of $\Tspl$, 
the distance $d(x,q(v_i))$ is bounded
above by $E\Lambda_n$, where $E$ is maximum number of edges in
a geodesic path in $\Tspl$.   

Set $R=26c\delta+12\delta$.
The distance between $q(v_i)$ and $s_i(q(v_i))$ is bounded 
above by $R$, by Proposition \ref{p: old-new label system on Thyp}.
It follows that 
$$d(x,s_i(x))\leq 2d(x,q(v_i))+R\leq R+2E\Lambda_n.$$
Since the value of $E$ is bounded above by $\binom{k}{2}$, we conclude that $x$ is $C_n(k)$-fixed by $S$, where 
$C_n(k)=R+2\binom{k}{2}\Lambda_n.$
\end{proof}

It follows that if all edges of $\Tspl$ have length less than or equal
to $\Lambda_n$, then
$B_\rho\leq C_n(k)$.

\begin{proof}[Proof of the Effective Compactness Theorem for Coxeter Groups]
Let $C_n(k)$ be the function defined in Lemma \ref{l: Tspl edge length bound}, and 
suppose that $B_\rho > C_n(k).$
Then $\H^n$ contains no point
$C_n(k)$-fixed by $S$.  
The useful subtree theorem (Theorem \ref{t: useful subtree}) 
guarantees that $\Tspl$ is nonempty, as $B_\rho>C_n(k)>R$.
Since $\Tspl$ is nonempty,
Lemma \ref{l: Tspl edge length bound} guarantees the existence
of an edge $e$ of $\Tspl$ whose shadow length is greater than $\Lambda_n$.
By Proposition \ref{prop: definition of special splitting}, 
the special splitting determined by $S^*(e)$
is nontrivial.  By Lemma \ref{c: e estimate}, this splitting is small.

Thus either $W$ admits a small special nontrivial splitting or 
$B_\rho\le C_n(k)$ for every $[\rho]\in \D(W,n)$.
\end{proof}


\end{document}